\documentclass[12pt,reqno]{amsart}
\usepackage[manuscripta,no-grey]{optional} 
\opt{grey}{
\usepackage{xcolor}
\pagecolor{black!80}
\color{white}
}

\usepackage{amssymb,bbm,graphicx}
\usepackage{comment}
\usepackage{enumitem}
\usepackage[colorlinks=true,linkcolor=blue!50,citecolor=blue!50,urlcolor=purple]{hyperref}

\makeatletter
\@namedef{subjclassname@2020}{\textup{2020} Mathematics Subject Classification}
\makeatother

\DeclareMathSymbol{\shortminus}{\mathbin}{AMSa}{"39} 

\usepackage{tikz}
\usepackage{tikz-cd}
\usepackage{xcolor}
\usetikzlibrary{arrows,decorations,decorations.markings,calc,positioning} 
\usepackage[all]{xy}

\usepackage{blkarray}
\newlength{\colwidth} 
\def\hspacing{14pt}
\BAnewcolumntype{D}{>{\centering$}p{\colwidth}<{$}@{\hspace{\hspacing}}}

\usepackage[margin=30mm]{geometry}

\usepackage{setspace}
\newcommand{\marginparstretch}{0.6}
\let\oldmarginpar\marginpar
\renewcommand\marginpar[1]{\-\oldmarginpar[\framebox{\setstretch{\marginparstretch}\begin{minipage}{\marginparwidth}{\raggedleft\tiny #1}\end{minipage}}]{\framebox{\setstretch{\marginparstretch}\begin{minipage}{\marginparwidth}{\raggedright\tiny #1}\end{minipage}}}}

\numberwithin{equation}{section}

\theoremstyle{plain}
\newtheorem{theorem}{Theorem}[section]
\newtheorem{proposition}[theorem]{Proposition}
\newtheorem{lemma}[theorem]{Lemma}
\newtheorem{corollary}[theorem]{Corollary}

\theoremstyle{remark}
\newtheorem{remark}[theorem]{Remark}

\theoremstyle{definition}
\newtheorem{definition}[theorem]{Definition}

\newtheorem{example}[theorem]{Example}
\newtheorem{notation}[theorem]{Notation}

\newcommand\cO{\mathcal{O}}

\newcommand\cA{\mathcal{A}}
\newcommand\cB{\mathcal{B}}
\newcommand\cC{\mathcal{C}}
\newcommand\cK{\mathcal{K}}

\newcommand\cT{\mathcal{T}}
\newcommand\cX{\mathcal{X}}
\newcommand\cY{\mathcal{Y}}
\newcommand\cV{\mathcal{V}}

\newcommand\dcat[1]{\mathrm{D}^b{#1}}
\renewcommand\setminus{-}

\newcommand\rootstack[3]{\sqrt[#1]{#2/#3}}
\newcommand\stack[1][]{\rootstack{#1}{X}{D}}
\newcommand\nth[1][n]{{#1}^{\text{th}}}
\newcommand\Hom{\textrm{Hom}}

\newcommand\quotcoord{y}
\newcommand\rootn{n}
\newcommand\stackbrack\big
\newcommand\sodbrack\big
\newcommand\sodcomma{,\,}
\newcommand\sodembed{\iota}

\title{Root stacks and periodic decompositions}
\author{A.\ Bodzenta, W.\ Donovan}

\address{Agnieszka Bodzenta:
	Institute of Mathematics, 
	University of Warsaw \\ Banacha 2 \\ Warsaw 02-097,
	Poland.} \email{A.Bodzenta@mimuw.edu.pl}

\address{Will Donovan: Yau MSC, Tsinghua University, Haidian District, Beijing 100084, China; BIMSA, Yanqi Lake, Huairou District, Beijing 101408, China.}
\email{donovan@mail.tsinghua.edu.cn}


\begin{document}

\begin{abstract}
For an effective Cartier divisor $D$ on a scheme $X$ we may form an $\nth[\rootn]$~root stack. Its derived category is known to have a semiorthogonal decomposition with components given by $D$ and $X$. We show that this decomposition is $2\rootn$-periodic. For~$\rootn=2$ this gives a purely triangulated proof of the existence of a known spherical functor, namely the pushforward along the embedding of~$D$. For~$\rootn>2$ we find a higher spherical functor in the sense of recent work of Dyckerhoff, Kapranov and Schechtman~\cite{DKS}. We use a realization of the root stack construction as a variation of GIT, which may be of independent interest.
\end{abstract}
\maketitle
\thispagestyle{empty}


\section{Introduction}

For an effective Cartier divisor $D$ on a scheme $X$ it is known that the derived pushforward functor from $\dcat{(D)}$ to $\dcat{(X)}$ is spherical. In this paper we explain how this spherical functor arises from a geometric 4-periodic semiorthogonal decomposition of~$\dcat{\stackbrack(\stack\stackbrack)}$, where $\stack$ is the square root stack for the divisor $D$. We furthermore exhibit a higher spherical functor, namely a \mbox{$2 \rootn$-spherical} functor, in this setting by giving an analogous $2\rootn$-periodic semiorthogonal decomposition of $\dcat{\stackbrack(\stack[\rootn]\stackbrack)}$ \mbox{for $n\geq 2$}. Along the way we improve on existing semiorthogonal decomposition results, in particular by working with the bounded coherent derived category, without smoothness assumptions.

\subsection{Spherical functors and periodic decompositions}
In the original formulation by R.~Anno and T.~Logvinenko~\cite{AnnoLogvinenko} a spherical functor is a DG functor $F\colon \mathrm{A} \to \mathrm{B}$ between pretriangulated DG categories with left and right adjoints~$L$ and~$R$, such that the cones of the $L\dashv F$ adjunction counit and unit induce autoequivalences of the homotopy categories of $\mathrm{A}$ and $\mathrm{B}$ respectively. In~\cite{HLShipman} D.~Halpern-Leistner and I.~Shipman proved that if $F$ is spherical then the category $\mathrm{T}$ glued from $\mathrm{A}$ and $\mathrm{B}$ via~$F$, see~\cite{KuzLun}, admits a 4-periodic semiorthogonal decomposition, see Definition \ref{def_periodic_SOD}.
\begin{align}\label{eqtn_4-per_SOD}
&\mathrm{T} = \langle \mathrm{A}, \mathrm{B} \rangle =  \langle \mathrm{B}, \mathrm{C} \rangle =  \langle \mathrm{C}, \mathrm{D} \rangle =  \langle \mathrm{D}, \mathrm{A} \rangle 
\end{align}
Moreover, given a 4-periodic semiorthogonal decomposition \eqref{eqtn_4-per_SOD}, the gluing functor $\sodembed_{\mathrm{A}}^R\sodembed_{\mathrm{B}}$ is spherical. These results suggest that one might consider an analogue of a spherical functor for triangulated categories without choosing a DG enhancement.

Similar ideas appeared in the work of  M.~Kapranov and V.~Schechtman who in~\cite{KapSche} noted that, given a 4-periodic semiorthogonal decomposition \eqref{eqtn_4-per_SOD} of a triangulated category $\mathrm{T}$ in a suitably-enhanced\footnote{For an explanation of the framework used for enhancement here, see \cite[Appendix~A]{KapSche}.} setting, the functor $\sodembed_{\mathrm{A}}^R\sodembed_{\mathrm{B}}$ is `spherical' in the sense that the cones of adjunction unit and counit are well-defined and are autoequivalences. Indeed, the functorial exact triangles for decompositions $\langle \mathrm{B}, \mathrm{C} \rangle $  and $\langle \mathrm{D}, \mathrm{A} \rangle$ respectively imply that the $\sodembed_{\mathrm{B}}^L{\sodembed_{\mathrm{A}}}\dashv \sodembed_{\mathrm{A}}^R {\sodembed_{\mathrm{B}}}$ adjunction unit and counit fit into functorial exact triangles as follows.
\begin{gather}
\sodembed_{\mathrm{A}}^R{\sodembed_{\mathrm{C}}}\sodembed_{\mathrm{C}}^R{\sodembed_{\mathrm{A}}} \to \textrm{Id}_{\mathrm{A}} \to \sodembed_{\mathrm{A}}^R{\sodembed_{\mathrm{B}}}\sodembed_{\mathrm{B}}^L{\sodembed_{\mathrm{A}}}\to \label{eqtn_cotwist_for_4-per}\\
\sodembed_{\mathrm{B}}^L{\sodembed_{\mathrm{A}}}\sodembed_{\mathrm{A}}^R{\sodembed_{\mathrm{B}}} \to \textrm{Id}_{\mathrm{B}} \to \sodembed_{\mathrm{B}}^L{\sodembed_{\mathrm{D}}} \sodembed_{\mathrm{D}}^L{\sodembed_{\mathrm{B}}} \to \label{eqtn_twist_for_4-per}
\end{gather}
The functorial cones of the unit and counit are compositions of mutation functors $\sodembed_{\mathrm{A}}^R{\sodembed_{\mathrm{C}}}$, $\sodembed_{\mathrm{C}}^R{\sodembed_{\mathrm{A}}}$, $\sodembed_{\mathrm{B}}^L{\sodembed_{\mathrm{D}}}$ and $\sodembed_{\mathrm{D}}^L{\sodembed_{\mathrm{B}}}$, hence they are equivalences by~\cite{Bon}. By analogy with this, given a \emph{triangulated} category $\mathrm{T}$ with a 4-periodic semiorthogonal decomposition as in~\eqref{eqtn_4-per_SOD}, we shall say that the gluing functor $\sodembed_{\mathrm{A}}^R\sodembed_{\mathrm{B}}$ is \emph{triangle-spherical}.

The first example of a geometric 4-periodic decomposition appeared in~\cite{BodBon}. There the appropriate quotient of the derived category of the fiber product of two varieties $X^+$ and $X^-$ related by a flop was proved to have a 4-periodic semiorthogonal decomposition with components $\dcat(X^{\pm})$ and $\dcat(\mathcal{A}_{f^{\pm}})$ for abelian null categories $\mathcal{A}_{f^{\pm}} \subset \textrm{Coh}(X^{\pm})$.

\subsection{Divisors and spherical functors}

A basic example of a spherical functor is $i_{D*} \colon \dcat(D)\to \dcat(X)$ for a scheme $X$ and the inclusion ${i_D} \colon D\to X$ of a Cartier divisor. This spherical functor appeared in the original paper of  R.~Anno~\cite{Anno} defining spherical functors and later in the work of N.~Addington~\cite{Addington}, though both statements were without proofs. Note that, in \cite{Anno} a DG enhancement was not present while in \cite{Addington} the author considered Fourier-Mukai kernels instead of functors between derived categories. In the case when both $D$ and $X$ are smooth, \mbox{A.~Bondal} and D.~Orlov~\cite{BonOrl} proved that the cone of the $i_D^* \dashv i_{D*}$ adjunction unit is an autoequivalence of $\dcat(X)$. Under the same assumptions, but allowing $D$ and~$X$ to be stacks, A.~Kuznetsov in~\cite[Proposition 3.4]{Kuz} proved the sphericity of~$i_{D*} $, compare~\cite[Lemma~2.8]{KuznetsovPerry}.

Without the smoothness assumption the first named author and A.~Bondal in~\cite[Theorem 3.1]{BodBon}  proved that, for unbounded derived categories of quasi-coherent sheaves, $i_{D*} \colon \mathrm{D}_{\textrm{QCoh}}(D) \to \mathrm{D}_{\textrm{QCoh}}(X)$ and its left adjoint form a spherical couple, i.e.~the adjunction unit and counit fit into functorial exact triangles as follows.
\begin{gather}
\big(\!- \otimes \,\cO_X(-D)\big) \to \textrm{Id} \to i_{D*}i_D^*\to\label{eqtn_cotw_for_divisor}\\
i_D^*i_{D*}  \to \textrm{Id} \to \big(\! - \otimes \,\cO_D(-D)[2]\big) \to \label{eqtn_tr_for_divisor}
\end{gather}
Note that a spherical couple, unlike a spherical functor, requires an adjoint on only one side. More precisely, it is a 2-categorical adjunction in the bicategory $\textbf{FM}$ whose objects are quasi-compact, quasi-separated schemes over a field while $\Hom_{\textbf{FM}}(X,Y)$ is the category $\mathrm{D}_{\textrm{QCoh}}(X \times Y)$. As the latter is triangulated, once a lifting of the adjunction (co)unit to a 2-morphism in $\textbf{FM}$ is fixed, one can consider the cone and check if it is an equivalence.

 \subsection{Root stacks and higher spherical functors}
 In this paper we consider the root stacks $\stack[\rootn]$ for $\rootn \geq 2$. Informally speaking, the root stack construction takes a scheme~$X$ and modifies it along an effective Cartier divisor~$D$ to get a stack with stabilizer groups $\mu_\rootn$ along $D$. As explained in Section~\ref{ssec_examples}, the stacky weighted projective line $\mathbb{P}(1,\rootn)$ is an instance of this construction, which becomes important and natural when passing from the setting of schemes to that of stacks.
 
  We prove that $\dcat{\stackbrack(\stack\stackbrack)}$ admits a 4-periodic semiorthogonal decomposition with components equivalent to $\dcat(D)$ and~$\dcat(X)$, with gluing functor $i_{D*}$ so that  in particular $i_{D*}$ is triangle-spherical, see below. More generally, we prove that $\dcat{\stackbrack(\stack[\rootn]\stackbrack)}$ admits a $2\rootn$-periodic semiorthogonal decomposition with $\rootn-1$ components equivalent to $\dcat(D)$ and one component equivalent to $\dcat(X)$. Such decompositions, without the periodicity statement, have previously been obtained  by A.~Ishii and K.~Ueda for smooth $D$ and~$X$~\cite[Theorem~1.6]{IshUed}, as well as by D.~Bergh, V.~Lunts and O.~Schn\"urer for perfect complexes~\cite[Theorem~4.7]{BLS}. The $2\rootn$-periodicity of a semiorthogonal decomposition for a root stack has previously been discussed by A.~Bondal~\cite{BondalTalk}, speaking on joint work with T.~Logvinenko.

In recent work, T.~Dyckerhoff, M.~Kapranov and V.~Schechtman define $N$-spherical functors of stable infinity categories~\cite{DKS}. In particular, their 4-spherical functors are analogues of spherical functors for DG categories. They prove that a semiorthogonal decomposition of a stable infinity category is $N$-periodic if and only if the gluing functor is $N$-spherical. Hence, one can think of the $2\rootn$-periodicity of the semiorthogonal decomposition of $\dcat{\stackbrack(\stack\stackbrack)}$ as $2\rootn$-sphericity of the gluing functor. 

\subsection{Variation of GIT}\label{sec.vgit} To construct the $2\rootn$-periodic decomposition we express the passage from $X$ to the root stack $\stack[\rootn]$ as a wall crossing in geometric invariant theory~(GIT), also known as a `variation of GIT'. There exists a general framework, see~\cite{HLShipman}, for associating spherical functors to \emph{balanced} wall crossings in GIT. The wall crossing in our case is not balanced, but for $\rootn=2$ a spherical functor still arises from it, as we shall see from the viewpoint of periodic decompositions.

 Previous work including~\cite{BFK,HL} has shown that stacks naturally arise when studying schemes and their derived categories. The present paper further illustrates this theme.

\subsection{Results}

For an effective Cartier divisor $D$ on a scheme $X$ we write $\sqrt[\rootn]{X/D}$ for the associated $\nth[\rootn]$~root stack. This may be described in a number of ways, which we explain in detail in Section~\ref{sect.rootstacks}. The most useful description for us is by a variation of~GIT for a $\mathbb{G}_m$-action on a scheme~$\cX$ defined as follows.

\begin{proposition}[Proposition~\ref{prop_quotients}]
	Write $\cX$ for the subscheme  $\{y z^\rootn = s\}$ of the total space of $\cO_X(D) \oplus \cO_X$ with  fiber coordinates $(y, z)$ and $s$ the canonical section of $\cO_X(D)$. Let $\mathbb{G}_m$ act fiberwise with weights $(-\rootn,1)$. Then the GIT quotients are as follows.
	\begin{equation*}
		[\cX^+/\mathbb{G}_m] \simeq \sqrt[\rootn]{X/D},\qquad 
		[\cX^-/\mathbb{G}_m] \simeq X.
	\end{equation*}

\end{proposition}

Base changing the root stack construction to the divisor $D$ itself, we obtain a square 
\begin{equation*}
\begin{tikzpicture}[scale=2]
\node (0) at (-1.5,0) {$D$};
\node (3) at (-1.5,1.5) {$\sqrt[\rootn]{\cO_D(D)}$};
\node (1) at (0,0)  {$X$};
\node (2) at (0,1.5) {$\sqrt[\rootn]{X/D}$};
\draw[->] (0) -- node[below]{$i_D$} (1);
\draw[->] (2) -- node[right]{$p$} (1);
\draw[->] (3) -- node[left]{$q$} (0);
\draw[->] (3) -- node[above]{$i$} (2);
\end{tikzpicture}
\end{equation*}
where $\sqrt[\rootn]{\cO_D(D)}$ is the root stack of the given line bundle on $D$, see Definition \ref{def_root_of_inv_sh}.

\begin{theorem}[Theorem \ref{thm_SOD_for_r_stack}] We have semiorthogonal decompositions as follows.
\begin{align*}
\dcat{\stackbrack(\sqrt[\rootn]{X/D}\stackbrack)} 
& = \sodbrack\langle \dcat(X)\sodcomma \dcat(D)\sodcomma \dcat(D)\otimes \cO\langle 1 \rangle\sodcomma \ldots\sodcomma \dcat(D) \otimes \cO\langle \rootn-2 \rangle \sodbrack\rangle \\
& = \sodbrack\langle \dcat(D)\sodcomma\dcat(X)\otimes \cO\langle 1 \rangle\sodcomma  \dcat(D)\otimes \cO\langle 1 \rangle\sodcomma \ldots\sodcomma \dcat(D) \otimes \cO\langle \rootn-2 \rangle \sodbrack\rangle \\
& = \dots \\
& = \sodbrack\langle\dcat(D)\sodcomma \dcat(D)\otimes \cO\langle 1 \rangle\sodcomma \ldots\sodcomma \dcat(D) \otimes \cO\langle \rootn-2 \rangle\sodcomma \dcat(X) \otimes \cO\langle \rootn-1 \rangle \sodbrack\rangle
\end{align*}
Here the bundle $\cO\langle 1 \rangle$ on $\sqrt[\rootn]{X/D}$ is induced by $\mathbb{G}_m$-weight~$1$, and the embeddings of  $\dcat{(X)}$ and~$\dcat{(D)}$ are respectively $p^*$ and~$i_* q^*$.
\end{theorem}
The last decomposition, twisted by $\cO\langle 1-\rootn \rangle$, has been previously obtained but in the smooth setting~\cite[Theorem~1.6]{IshUed} or for perfect complexes~\cite[Theorem~4.7]{BLS}.

 \begin{remark} The similarity of these decompositions with the Orlov decomposition for a blowup provides part of the justification for thinking of $\stack[\rootn]$ as a stacky `blowup in codimension one'.\end{remark}

Our main theorem is then the following.

\begin{theorem}[Theorem \ref{thm_2r_period}]
	Take a full subcategory
	\[\mathcal{D} = \sodbrack\langle \dcat(D)\sodcomma \dcat(D) \otimes \cO\langle 1\rangle\sodcomma \ldots\sodcomma \dcat(D) \otimes \cO\langle \rootn-2 \rangle \sodbrack\rangle \subset \dcat{\stackbrack(\stack[\rootn]\stackbrack)}.\] 
	Then the semiorthogonal decomposition
	\begin{equation*}
		\dcat{\stackbrack(\stack[\rootn]\stackbrack)} = \langle \dcat(X), \mathcal{D} \rangle
	\end{equation*}
is $2\rootn$-periodic, in the sense that its $\nth[2\rootn]$ right dual decomposition is the original decomposition, see Definition \ref{def_periodic_SOD}. 

\end{theorem}

In particular, for $n=2$ we get a 4-periodic decomposition as in~\eqref{eqtn_4-per_SOD} 
\begin{equation}\label{eqtn_main_SOD} 
\begin{aligned} 
\dcat{\stackbrack(\stack \stackbrack)}  &= \sodbrack\langle \dcat(X) \sodcomma  \dcat(D) \sodbrack\rangle \\
&= \sodbrack \langle \dcat(D) \sodcomma \dcat(X)\otimes \cO\langle 1 \rangle \sodbrack \rangle\\
& = \sodbrack\langle \dcat(X) \otimes \cO\langle 1 \rangle \sodcomma  \dcat(D) \otimes \cO\langle 1 \rangle \sodbrack\rangle \\
 &= \sodbrack \langle \dcat(D)\otimes \cO\langle 1 \rangle \sodcomma
 \dcat(X)\sodbrack \rangle. 
\end{aligned}
\end{equation} 

We calculate the gluing functors for the above decompositions in Proposition \ref{prop_gluing_functor}. In~particular, we have the following.
\begin{corollary}[Corollary \ref{cor_div_is_sph}] The functor $i_{D*}\colon \dcat{(D)} \to \dcat{(X)}$ is triangle-spherical.
\end{corollary}

\subsection*{Notation} We write $\sqrt[\rootn]{X/D}$ for the $\nth[\rootn]$~root stack associated to a scheme $X$ with an effective Cartier divisor $D$. For a noetherian scheme or stack $X$ we denote by $\dcat(X)$ the bounded derived category of coherent sheaves on $X$. For the quotient stack of a scheme $\mathcal{X}$ with an action of~$G$ we write $[\mathcal{X}/G]$ or simply~$\mathcal{X}/G$. 

\subsection*{Acknowledgements}

The first author was partially supported by Polish National Science Centre grants 2018/31/D/ST1/03375 and 2021/41/B/ST1/03741. The second author was supported by Yau MSC, Tsinghua University, Yanqi Lake BIMSA, and the Thousand Talents Plan.
The authors are grateful for discussions with N.~Addington, R.~Anno, A.~Bondal, J.~Jelisiejew, T.~Kuwagaki, T.~Logvinenko and E.~Segal, and for helpful comments from an anonymous referee.

\section{Root stacks as variation of GIT}\label{sect.rootstacks}

We give definitions and a number of descriptions of the root stack.

\subsection{Root stacks}\label{sect.rootstackdef}
	
	Take an algebraically closed field $k$ of characteristic zero and a reduced separated noetherian $k$-scheme  $X$, with an effective Cartier divisor $D$ given on an open cover by~$\{U_i,f_i\}$. We denote by $s$ the canonical section of $\cO_X(D)$.

	Recall  that the quotient stack $[\mathbb{A}^1_k/\mathbb{G}_m]$ represents the functor which to a scheme~$B$ assigns the groupoid of generalized Cartier divisors $(L,\, \rho \colon L \to \cO_B)$ on $B$ with isomorphisms preserving the morphisms to the structure sheaf, see for instance \cite[Proposition 10.3.7]{Olss}. Let $\delta\colon X\to [\mathbb{A}^1_k/\mathbb{G}_m]$ be the morphism given by \[\big(\cO_X(-D),\,s(-D) \colon \cO_X(-D) \to \cO_X\big).\]
	
	\begin{definition}\cite{Cad, AbGrVi} The $\nth[\rootn]$ \emph{root stack} $\sqrt[\rootn]{X/D}$ for $\rootn\geq 1$ is given by the following fiber product, where $e_\rootn$ is induced by taking the $\nth[\rootn]$ power.
\begin{equation*}
    \begin{tikzpicture} [scale=2.5,xscale=1.25]
   \node (A) at (0,0) {$X$};
   \node (B) at (1,0) {$[\mathbb{A}^1_k/\mathbb{G}_m]$};
   \node (C) at (0,1) {$\stack[\rootn]$}; 
   \node (D) at (1,1) {$[\mathbb{A}^1_k/\mathbb{G}_m]$}; 
      \draw [->] (A) to node[below] {$\delta$} (B);
      \draw [->] (D) to node[right] {$e_\rootn$} (B);
      \draw [->] (C) to (D);
      \draw [->] (C) to node[left] {$p$} (A);
  \end{tikzpicture}
\end{equation*}
\end{definition}

\begin{notation} Write\, $\stack$ for the $\rootn=2$ case, namely the \emph{square root stack}.
\end{notation}

For a scheme $B$, objects of the groupoid $\sqrt[\rootn]{X/D}(B)$ are morphisms $f \colon B \to X $ and sections $t\colon \cO_B \to \mathcal{L}$ of invertible sheaves on $B$ together with an isomorphism $\mathcal{L}^{\otimes \rootn} \xrightarrow{\sim} f^*\cO_X(D)$ which identifies $t^{\otimes \rootn}$ with $f^*(s)$. Morphisms are isomorphisms of invertible sheaves $\mathcal{L} \xrightarrow{\sim} \mathcal{L}'$ commuting with all the additional data.

Recall that any invertible sheaf gives a $\mathbb{G}_m$-torsor and any $\mathbb{G}_m$-torsor can be obtained in this way. Hence,  the  stack $B\mathbb{G}_m$ represents the functor which to a scheme $B$ assigns the groupoid of invertible sheaves on $B$. Given $\mathcal{M} \in \textrm{Pic}(X)$ let $\mu\colon X\to B\mathbb{G}_m$ be the morphism given by $\mathcal{M}$.

\begin{definition}\cite{AbGrVi} \label{def_root_of_inv_sh}The $\nth[\rootn]$ \emph{root stack of a line bundle} $\sqrt[\rootn]{\mathcal{M}}$ for $\rootn\geq 1$ is given by the following fiber product, where $e_\rootn$ is induced by taking the $\nth[\rootn]$ power.
	\begin{equation*}
	\begin{tikzpicture} [scale=2.5,xscale=1.25]
	\node (A) at (0,0) {$X$};
	\node (B) at (1,0) {$B\mathbb{G}_m$};
	\node (C) at (0,1) {$\sqrt[\rootn]{\mathcal{M}}$}; 
	\node (D) at (1,1) {$B\mathbb{G}_m$}; 
	\draw [->] (A) to node[below] {$\mu$} (B);
	\draw [->] (D) to node[right] {$e_\rootn$} (B);
	\draw [->] (C) to (D);
	\draw [->] (C) to node[left] {$q$} (A);
	\end{tikzpicture}
	\end{equation*}
\end{definition}

For a base scheme $B$, an object of the groupoid $\sqrt[\rootn]{\mathcal{M}}(B)$ is a  morphism $f\colon B\to X$ and an invertible sheaf $\mathcal{L}$ on $B$ together with an isomorphism $\mathcal{L}^{\otimes \rootn} \xrightarrow{\sim} f^*\mathcal{M}$. Morphisms are isomorphisms of invertible sheaves $\mathcal{L} \xrightarrow{\sim} \mathcal{L}'$ commuting with the additional data.

\subsection{Quotient description of root stacks}

	The stack $\sqrt[\rootn]{X/D}$ can be also viewed as a quotient stack, following \cite{AbGrVi}. Write $\cT$ for the total space of $\cO_X(D)$ and $\cT^\circ$ for the total space of the associated $\mathbb{G}_m$-bundle, namely $\cT$ with the zero section removed. On $\cT^\circ \times \mathbb{A}^1_k$ consider the $\mathbb{G}_m$-action as follows.
	$$
	\lambda\cdot (u,z) = (\lambda^{-\rootn}u, \lambda z)
	$$
	Note that the morphism $\cT^\circ \times \mathbb{A}^1_k \to \cT$, $(u,z) \mapsto u z^\rootn$ is $\mathbb{G}_m$-invariant for the trivial $\mathbb{G}_m$-action on $\cT$, hence it gives a morphism $[\cT^\circ \times \mathbb{A}^1_k/\mathbb{G}_m] \to \cT$.

	\begin{proposition}\cite[Appendix  B.2]{AbGrVi}\label{prop_global_dscrp_of_root}
		Consider the canonical section $s$ of $\cO_X(D)$ as a subscheme of the total space~$\cT$, and let $\cV \subset \cT^\circ \times \mathbb{A}^1_k$ be its inverse image. Then \[ [\cV/\mathbb{G}_m] \simeq \sqrt[\rootn]{X/D}.\] 
	\end{proposition} 
\begin{proof}
Let $u_i$ be a local coordinate on $\cT^\circ$, $z$ a coordinate on $\mathbb{A}^1_k$, and $\{U_i,f_i\}$ a representative of $D$. Then $\cV$ is glued from $\cV_i = \{u_i z^\rootn = f_i\}\subset \cT^\circ|_{U_i} \times \mathbb{A}^1_k$. Let $p \colon \cV \to X$  be the canonical projection.

Note that $z$ is a section of  the trivial line bundle $\cO_{\cV}\langle -1 \rangle$  with $\mathbb{G}_m$-weight~$-1$. \mbox{Functions $u_i$} glue to an isomorphism $\quotcoord\colon \cO_{\cV}\langle - \rootn \rangle\xrightarrow{\sim} \mathcal{N}$, for  a line bundle $\mathcal{N}$ with trivial \mbox{$\mathbb{G}_m$-action}. The composition $\cO_{\cV} \xrightarrow{z^{\otimes \rootn}} \cO_{\cV}\langle -\rootn \rangle \xrightarrow{\quotcoord} \mathcal{N}$ is, by the definition of $\cV$, the pullback $p^*(s)$ of the canonical section $s$ of $\cO_X(D)$. Hence, $\mathcal{N} \simeq p^*\cO_X(D)$ and $z$~defines a section of a line bundle on $\cV$ whose $\nth[\rootn]$ power is identified by the isomorphism~$\quotcoord$ with~$p^*(s)$.

To define a morphism $[\cV/\mathbb{G}_m]\to  \sqrt[\rootn]{X/D}$ consider a scheme $B$ and a morphism $B \to [\cV/\mathbb{G}_m]$, i.e.\  take $\pi$ a principal $\mathbb{G}_m$-bundle and $h$ a $\mathbb{G}_m$-equivariant morphism fitting into a commutative diagram as follows.
\begin{equation*}
\begin{tikzpicture} [scale=2]
\node (A) at (0,0) {$B$};
\node (B) at (1,0) {$X$};
\node (C) at (0,1) {$C$}; 
\node (D) at (1,1) {$\cV$}; 
\draw [->] (A) to node[below] {$f$} (B);
\draw [->] (D) to node[right] {$p$} (B);
\draw [->] (C) to node[above] {$h$} (D);
\draw [->] (C) to node[left] {$\pi$} (A);
\end{tikzpicture}
\end{equation*}
Note that $f$ exists as $p \circ h$ is $\mathbb{G}_m$-invariant and $\pi$ is the categorical quotient. Pullback of $z$ along $h$ yields a morphism $h^*(z) \colon \cO_C \to h^*\cO_{\cV}\langle -1 \rangle$ of $\mathbb{G}_m$-linearized line bundles on~$C$. As $\pi^*\colon \textrm{Pic}(B) \xrightarrow{\sim}\textrm{Pic}_{\mathbb{G}_m}(C)$ is an isomorphism, see  \cite[Proposition 3.3.1]{Brion}, there exists a morphism $t\colon \cO_B \to \mathcal{L}$ in $\textrm{Pic}(B)$ which pulls back to $h^*(z)$ along~$\pi$. Similarly, there exists an isomorphism $\widetilde{\quotcoord}\colon \mathcal{L}^{\otimes \rootn} \to f^*\cO_X(D)$ which pulls back to $h^*(\quotcoord) \colon h^*\cO_{\cV}\langle -\rootn \rangle \to h^*p^* \cO_X(D)$. Then $\widetilde{\quotcoord} \circ t^{\otimes \rootn}$ is the pullback via $f$ of the canonical section $s$ of $\cO_X(D)$. Indeed, applying the  isomorphism $\pi^*$, one gets $\pi^*(\widetilde{\quotcoord} \circ t^{\otimes \rootn}) = h^*(\quotcoord\circ z^{\otimes \rootn}) =h^*p^*(s) = \pi^*f^*(s)$.

In the opposite direction, let $f\colon B \to X$ be a morphism of schemes, $t \colon \cO_B \to \mathcal{L}$ a morphism of invertible sheaves and $a \colon \mathcal{L}^{\otimes \rootn} \xrightarrow{\sim} f^*\cO_X(D)$ an isomorphism such that $a \circ t^{\otimes \rootn} = f^*(s)$. Let $\pi\colon C\to B$ be the $\mathbb{G}_m$-torsor associated with $\mathcal{L}^{-1}$. Consider $\widetilde{h} \colon C\to \textrm{Tot}(\cO_X(D) \oplus \cO_X)$ given by an element of $\Hom(\cO_C, \pi^*f^*(\cO_X(D) \oplus \cO_X))$ with components
\begin{align*}
a\otimes (\mathcal{L}^{-1})^{\otimes \rootn} & \in \,\Hom(\cO_B, (\mathcal{L}^{-1})^{\otimes \rootn} \otimes f^* \cO_X(D)) \\
& \subset \bigoplus_{k\in \mathbb{Z}} \Hom(\cO_B, \mathcal{L}^{\otimes k} \otimes f^* \cO_X(D)) \\
& \simeq\Hom(\cO_B, \pi_* \cO_C \otimes f^*\cO_X(D))  \\
& \simeq \Hom(\cO_C, \pi^*f^*\cO_X(D)),
\end{align*} 
and
\begin{align*}
t  & \in\, \Hom(\cO_B, \mathcal{L}) \\
&\subset\bigoplus_{k\in \mathbb{Z}}\Hom(\cO_B, \mathcal{L}^{\otimes k}) \\
& \simeq \Hom(\cO_B, \pi_* \cO_C) \\ 
& \simeq \Hom(\cO_C, \cO_C).
\end{align*} 
The morphism $\widetilde{h}$ induces a $\mathbb{G}_m$-equivariant morphism $h\colon C\to \cV$. Indeed, in local coordinates $(b,u) \in B \times \mathbb{G}_m$, the morphism $\widetilde{h}$ is given by $(b,u)\mapsto (f(b), u^{-\rootn}a(b), u t(b))$. Hence,  $h$~corresponds to  $B \to [\cV/\mathbb{G}_m]$. 

It is straightforward to check that the above constructions yield mutually inverse maps $[\cV/\mathbb{G}_m]\leftrightarrow \sqrt[\rootn]{X/D}$.
\end{proof}
	
\begin{remark}
	The proof of Proposition \ref{prop_global_dscrp_of_root} implies that $\cO_{\cV}\langle -\rootn \rangle$ is isomorphic to the pullback of $\cO_X(D)$.
\end{remark}

To view $\sqrt[\rootn]{\mathcal{M}}$ as a global quotient stack, we have the following, by a similar argument to Proposition~\ref{prop_global_dscrp_of_root}.

\begin{proposition}\label{prop_glob_quot_for_line_bundle}
	Consider the $\mathbb{G}_m$-bundle $\mathcal{M}^\circ$ associated to~$\mathcal{M}$ with the $\mathbb{G}_m$-action $\lambda \cdot u = \lambda^{-\rootn} u$. Then
	\[[\mathcal{M}^\circ/\mathbb{G}_m] \simeq \sqrt[\rootn]{\mathcal{M}}.\]
\end{proposition}	
\begin{proof}
	Let $p\colon \mathcal{M}^\circ\to X$ be the canonical projection. From the definition of $\mathcal{M}^\circ$ it follows that there exists an isomorphism $y \colon \cO_{\mathcal{M}^\circ}\langle -\rootn  \rangle \xrightarrow{\sim} p^*\mathcal{M}$.
	
	To define a morphism $[\mathcal{M}^\circ/\mathbb{G}_m] \to \sqrt[\rootn]{\mathcal{M}}$ consider a scheme $B$ and a morphism $B \to [\mathcal{M}^\circ/\mathbb{G}_m]$, i.e.\ take $\pi$ a principal $\mathbb{G}_m$-bundle and $h$ a $\mathbb{G}_m$-equivariant morphism fitting into a commutative diagram as follows.
	\begin{equation*}
	\begin{tikzpicture} [scale=2]
	\node (A) at (0,0) {$B$};
	\node (B) at (1,0) {$X$};
	\node (C) at (0,1) {$C$}; 
	\node (D) at (1,1) {$\mathcal{M}^\circ$}; 
	\draw [->] (A) to node[below] {$f$} (B);
	\draw [->] (D) to node[right] {$p$} (B);
	\draw [->] (C) to node[above] {$h$} (D);
	\draw [->] (C) to node[left] {$\pi$} (A);
	\end{tikzpicture}
	\end{equation*}
	Note that $f$ exists as $p \circ h$ is $\mathbb{G}_m$-invariant and $\pi$ is the categorical quotient.
Then $h^*(\cO_{\mathcal{M}^\circ}\langle -1 \rangle)$ is a $\mathbb{G}_m$-linearized line bundle on $C$, and further there is an isomorphism $h^*(y) \colon h^*(\cO_{\mathcal{M}^\circ}\langle- 1 \rangle)^{\otimes \rootn}\xrightarrow{\sim} h^*p^*\mathcal{M}$. It follows that there exists $\mathcal{L}\in\textrm{Pic}(B)$ with an isomorphism $\widetilde{y} \colon \mathcal{L}^{\otimes \rootn} \xrightarrow{\sim} f^*(\mathcal{M})$. Indeed, $\mathcal{L}$ and $\widetilde{y}$ are the preimages under the isomorphism $\pi^*\colon \textrm{Pic}(B) \xrightarrow{\sim}\textrm{Pic}_{\mathbb{G}_m}(C)$, see \cite[Proposition~3.3.1]{Brion},  of $h^* (\cO_{\mathcal{M}^\circ}\langle - 1 \rangle)$ and $h^*(y) \colon  h^*(\cO_{\mathcal{M}^\circ}\langle -1 \rangle)^{\otimes \rootn} \to h^*p^*(\mathcal{M}) \simeq \pi^*f^*(\mathcal{M})$.

	In the opposite direction, let $f\colon B \to X$ be a morphism of schemes, $\mathcal{L}\in \textrm{Pic}(B)$ and $a \colon \mathcal{L}^{\otimes \rootn} \xrightarrow{\sim} f^*\mathcal{M}$. Let $\pi\colon C\to B$ be the $\mathbb{G}_m$-torsor associated with $\mathcal{L}$. Consider $\widetilde{h} \colon C\to \textrm{Tot}(\mathcal{M})$ given by the following element of $\Hom(\cO_C, \pi^*f^*\mathcal{M})$.
	\begin{align*}
	a\otimes (\mathcal{L}^{-1})^{\otimes \rootn} & \in \,\Hom(\cO_B, (\mathcal{L}^{-1})^{\otimes \rootn} \otimes f^* \mathcal{M}) \\
	& \subset \bigoplus_{k\in \mathbb{Z}} \Hom(\cO_B, \mathcal{L}^{\otimes k} \otimes f^* \mathcal{M}) \\
	& \simeq\Hom(\cO_B, \pi_* \cO_C \otimes f^*\mathcal{M})  \\
	& \simeq \Hom(\cO_C, \pi^*f^*\mathcal{M})
	\end{align*} 
	The morphism $\widetilde{h}$ induces a $\mathbb{G}_m$-equivariant morphism $h\colon C\to \mathcal{M}^\circ$. Indeed, in local coordinates $(b,u) \in B \times \mathbb{G}_m$,  the morphism $\widetilde{h}$ is given by $(b,u)\mapsto (f(b), u^{-\rootn}a(b))$. Hence,  $h$~corresponds to  $B \to [\mathcal{M}^\circ/\mathbb{G}_m]$. 
	
	It is straightforward to check that the above constructions yield mutually inverse maps $[\mathcal{M}^\circ/\mathbb{G}_m]\leftrightarrow \sqrt[\rootn]{\mathcal{M}}$.
\end{proof}
	
	\begin{remark}\label{rem_O_r_for_inv_sh}
		The proof of Proposition \ref{prop_glob_quot_for_line_bundle} implies that $\cO_{\mathcal{M}^\circ}\langle -\rootn \rangle$ is isomorphic to the pullback of $\mathcal{M}$.
	\end{remark}

\subsection{GIT description of the root stack $\stack[\rootn]$}\label{subsec_GIT_of_root}

For a scheme $\cX$ with a $\mathbb{G}_m$-action, we denote the fixed locus by $Z$. Then  
\[
S^+= \big\{x\in \cX\,|\, \lim_{\lambda\to 0 } \lambda\cdot x \in Z\big\} \qquad\text{and} \qquad
S^-= \big\{x\in \cX\,|\, \lim_{\lambda\to 0} \lambda^{-1}\cdot x \in Z\big\}
\]
are the unstable loci. Write $\mathcal{X}^\pm = \mathcal{X} \setminus S^\pm$ for the semistable loci.  Note that here we allow finite stabilizers in these loci, so that the GIT quotients $[\cX^\pm/\mathbb{G}_m]$ may be Deligne--Mumford stacks, not necessarily schemes.

\smallskip

To construct the root stack as a GIT quotient, we make the following definition. This construction previously appeared, in a certain dimension~$2$ example for~$n=2$, in work of the second named author and T.~Kuwagaki~\cite[Proposition~5, second case]{DonovanKuwagaki}.

\begin{definition}\label{def.quotforroot} Let $\cT_\rootn$ be the total space of the rank 2 vector bundle $\cO_X(D) \oplus \cO_X $ with fiber coordinates $(y, z)$, and a fiberwise $\mathbb{G}_m$-action with weights $(-\rootn,1)$. Let $\cX_\rootn$ be the hypersurface given by \[\cX_\rootn = \{y z^\rootn = s\} \subset \cT_\rootn\] where $s$ is the canonical section of $\cO_X(D)$.\end{definition}

The equation $y z^\rootn = s$ is valued in $\cO_X(D)$ and is $\mathbb{G}_m$-invariant by construction. We explicitly describe $\cX_\rootn$ using the representative $\{U_i, f_i\}$ of $D$.  Classically, we can view the invertible sheaf $\cO_X(D)$ as an ideal sheaf in the sheaf $\cK_X$ of total fractions. Namely, take $\cO_X(D)$ to be the sub-$\cO_X$-module of  $\cK_X$ generated by $f_i^{-1}$ on $U_i$ \cite[Chapter~II.6]{Har}. Then multiplication with $f_i$ is an isomorphism $\zeta_i \colon \cO_{U_i}(D) \to \cO_{U_i}$. The invertible sheaf $\cO_X(D)$ has a canonical section $s = 1\in \Gamma(X, \cO_X(D))$. Note that $s|_{U_i}$ corresponds under~$\zeta_i$ to the regular function $f_i \in \Gamma(U_i, \cO_X)$. The trivializations $\zeta_i$ allow us to consider fiber coordinates $(y_i, z_i)$ on $\mathcal{T}_\rootn|_{U_i}$. The scheme $\cX_\rootn$ is then glued from $\{y_iz_i^\rootn = f_i\}\subset \mathcal{T}_\rootn|_{U_i}$.

\smallskip

Noting that the $\mathbb{G}_m$-action on $\cT_\rootn$ restricts to $\cX_\rootn$, we now describe the GIT quotients for~$\cX_\rootn$. 

\begin{proposition}\label{prop_quotients}
	Consider $\cX_\rootn$ from Definition~\ref{def.quotforroot} with its $\mathbb{G}_m$-action. The associated GIT quotients are as follows.
	\begin{itemize}
		\item[$(+)$] We have $[\cX_\rootn^+/\mathbb{G}_m] \simeq \sqrt[\rootn]{X/D}$. The unstable locus $S^+=\{y =0\}$ is isomorphic to the total space of $\cO_D$ with fiberwise $\mathbb{G}_m$-action of weight $1$.
		\item[$(-)$] We have $[\cX_\rootn^-/\mathbb{G}_m] \simeq X$. The unstable locus $S^-=\{z =0\}$ is isomorphic to the total space of $\cO_D(D)$ with fiberwise $\mathbb{G}_m$-action of weight $-\rootn$.
	\end{itemize}
The fixed locus $Z = \{ y, z = 0 \}$ is isomorphic to $D$.
\end{proposition}

\begin{proof}
	We first describe the unstable loci for $\cX_\rootn$, which are given by restricting the unstable loci for $\cT_\rootn$. Since the $\mathbb{G}_m$-action on $\cT_\rootn$ is fiberwise, the unstable loci for a given linearization can be computed for the action of $\mathbb{G}_m$ on $\mathbb{A}^2_k$ with weights $(-\rootn,1)$ and coordinates $(y,z)$. The fixed locus of this action is $(0,0)$, with unstable loci $\{y = 0\}$ and $\{z =0\}$. For  $\mathcal{T}_n$ therefore the fixed locus $Z$ is the zero section, and the unstable loci over the open set~$U_i$ are $S^+= \{y_i = 0\}$ and $S^- =\{z_i =0\}$, so the description of the unstable loci follows.

	The rest of the $(+)$~side follows from an isomorphism of $\cX_\rootn^+$ and $\cV$ as \mbox{$\mathbb{G}_m$-schemes}, see Proposition \ref{prop_global_dscrp_of_root}. For the $(-)$~side, note that $\cX^-_\rootn$ is a $\mathbb{G}_m$-torsor over $X$. Indeed, $\mathbb{G}_m$~acts freely on the total space of the $\mathbb{G}_m$-bundle associated to $\cO_X$ and the coordinate~$y_i$ along the fiber of $\cO_X(D)$ is uniquely determined by a point of $X$ and the value of $z_i$ in the fiber over it. It follows that $X$ is the geometric quotient of $\cX_\rootn^-$ by $\mathbb{G}_m$, hence  a categorical one~\cite[Proposition 0.1]{MumFogKir}. As $[\cX_\rootn^-/\mathbb{G}_m]$ is the categorical quotient too, the required isomorphism follows from the uniqueness of the quotient.
\end{proof}

\begin{remark}\label{rem_O(-r)}
	Combining the above proposition and the proof of Proposition \ref{prop_global_dscrp_of_root} gives that $ \cO_{\cX_\rootn^+/\mathbb{G}_m}\langle -\rootn \rangle \simeq \cO_{\sqrt[\rootn]{X/D}}\langle -\rootn \rangle$ is isomorphic to the pullback of $\cO_X(D)$.
\end{remark}

\subsection{Toric example}\label{ssec_examples}

For concreteness, we give the following.

\begin{example} Let $X = \mathbb{P}^1$ with coordinates $(x_0:x_1)$, and let $D$ be the point~$(1:0)$ so that $\cO_X(D) \simeq \cO_X(1)$. Then $\stack[\rootn]$ is isomorphic to the stacky weighted projective line $\mathbb{P}(1,\rootn)$. To see this, note first that $X$ may be viewed as $[\mathbb{A}^2_k/\mathbb{G}_m] \setminus \{x_0,x_1=0\}$ where coordinates and $\mathbb{G}_m$-weights are as follows.
\settowidth{\colwidth}{$-\rootn$} 
\[\begin{blockarray}{DD}
\small{x_0}&\small{x_1}&\\[1ex]
\begin{block}{(DD)}
1&1&\\
\end{block}\end{blockarray}\]
Then $\cX_\rootn$ with its $\mathbb{G}_m$-action may be presented as $\{ x_1 = yz^\rootn\} \subset [\mathbb{A}^4_k/\mathbb{G}_m^2] \setminus \{x_0,x_1=0\}$ with weights
\[\begin{blockarray}{DDDD}
\small{x_0}&\small{x_1}&\small{y}&\small{z}&\\[1ex]
\begin{block}{(DDDD)}
1&1&1&0&\\
0&0&-\rootn&1&\\
\end{block}\end{blockarray}\]
and $\mathcal{X}_\rootn^+$ is given by removing $\{y=0\}$. Changing basis in the torus $\mathbb{G}_m^2$ corresponds to row operations, so we may take weights as follows. 
\[\begin{blockarray}{DDDD}
\small{x_0}&\small{x_1}&\small{y}&\small{z}&\\[1ex]
\begin{block}{(DDDD)}
1&1&1&0&\\
\rootn&\rootn&0&1&\\
\end{block}\end{blockarray}\]
Now setting $y=1$ using the $\mathbb{G}_m$ factor corresponding to the first row, we deduce that $[\mathcal{X}_\rootn^+/\mathbb{G}_m] \simeq \{x_1 = z^\rootn\} \subset [\mathbb{A}^3_k/\mathbb{G}_m] \setminus \{x_0,x_1=0\}$ with weights below.
\[\begin{blockarray}{DDDDD}
&\small{x_0}&\small{x_1}&\small{z}&&\\[1ex]
\begin{block}{D(DDD)D}
&\rootn&\rootn&1&&\\
\end{block}\end{blockarray}\]
This is isomorphic to $[\mathbb{A}^2_k/\mathbb{G}_m] \setminus \{y_0,y_1=0\}$ with weights as follows, giving the claim.
\[\begin{blockarray}{DDDD}
&&\small{y_0}&\small{y_1}&&\\[1ex]
\begin{block}{D(DD)D}
&&\rootn&1&&\\
\end{block}\end{blockarray}\]

\end{example}

\section{Derived category of the root stack}

We first explain general theory, before applying to our setting.

\subsection{Derived category of GIT quotients}
\label{section.derivedcatGIT}

Take $\mathbb{G}_m$ acting on a scheme $\cX$ with fixed locus $Z$ and unstable loci $S^\pm$ as in Section \ref{subsec_GIT_of_root}. Recall that $\mathcal{X}^\pm = \mathcal{X} \setminus S^\pm$, and let $i_\pm \colon \cX^\pm \to \cX$ be the inclusions. Write $\pi_{\pm}\colon S^{\pm} \to Z$ for the maps which send $x$ to
\[\lim_{\lambda\to 0}\lambda\cdot x \qquad \text{and} \qquad \lim_{\lambda\to 0}\lambda^{-1}\cdot x\] respectively. 

\begin{proposition}\cite[Lemma 2.9, Theorem 2.10]{HL}\label{prop_HL_equiv}
	Assume 
	\begin{itemize}
		\item[$(A)$] $\pi_{\pm} \colon S^{\pm} \to Z$ are locally trivial bundles of affine spaces, and
		\item[$(R)$] the inclusions $S^{\pm}\to \cX$ are regular embeddings.
	\end{itemize}
	Under assumption $(R)$ the derived restriction along the closed immersions $Z\to S^{\pm}$ of the relative cotangent complex $L^\bullet_{S^{\pm}|\mathcal{X}}$ is $\det \mathcal{N}^\vee_{S^{\pm}|\mathcal{X}}|_Z[1]$. We make the following further assumption.
	\begin{itemize}
		\item[$(L)$] $\det \mathcal{N}_{S^{\pm}|\mathcal{X}}^{\mp 1}|_Z$ has positive $\mathbb{G}_m$-weight $\eta_\pm$.
	\end{itemize}	
Then the derived restriction functor $i_\pm^*$ gives an equivalence of 
\begin{equation}\label{eqtn_def_of_C_omega}
\mathcal{C}_{[\omega, \omega+\eta_{\pm})}= \big\{ E \in \dcat(\mathcal{X}/\mathbb{G}_m) \,\big|\, \mathcal{H}^\bullet ( i_Z^* E) \textrm{ have weights in $[\omega, \omega+\eta_{\pm})$} \big\}
\end{equation}
and $\dcat(\mathcal{X}^{\pm}/\mathbb{G}_m)$ where $i_Z$ is the closed immersion $Z \to \mathcal{X}$.
\end{proposition}
\begin{proof}
	By  \cite[Theorem 2.10]{HL} the conditions $(A)$ and $(L)$ imply an equivalence of $\dcat(\cX^{\pm}/\mathbb{G}_m)$ with a subcategory  $\mathcal{G}_\omega$ of $\dcat(\cX/\mathbb{G}_m)$ given by
\begin{equation*}
\mathcal{G}_\omega= \left\{ E \in \dcat(\mathcal{X}/\mathbb{G}_m) \,\middle|\, \begin{array}{l} \mathcal{H}^\bullet ( k_\pm^* j_\pm^* E) \textrm{ have weights $\geq \omega$} \\ \mathcal{H}^\bullet ( k_\pm^* j_\pm^! E) \textrm{ have weights $< \omega$} \end{array}
 \right\}
\end{equation*}
where we notate morphisms as follows.
\[\begin{tikzpicture}[scale=1.75]
\node (0) at (-1.25,0) {$Z$};
\node (1) at (0,0)  {$S^\pm$};
\node (2) at (1.25,0) {$\mathcal{X}$};
\draw[right hook->] (0) -- node[below]{$k_\pm$} (1);
\draw[right hook->] (1) -- node[below]{$j_\pm$} (2);
\end{tikzpicture}\]
Then the argument  in the proof of \cite[Lemma 2.9]{HL} shows that if $(R)$ holds the category  $\mathcal{G}_\omega$ can be described as $\mathcal{C}_{[\omega, \omega+\eta_{\pm})}$ in \eqref{eqtn_def_of_C_omega}.
\end{proof}

\begin{remark} The positive integer $\eta_{\pm}$ is known as the window width. It is calculated in our setting in Proposition \ref{prop_window_width} below.\end{remark}

Under the assumptions $(A)$, $(R)$ and $(L)$ the references \cite{BFK,HL} furthermore give semiorthogonal decompositions of $\mathcal{C}_{[0,\eta_\pm+l)}$ for a positive integer~$l$. For simplicity of notation, we give the $(-)$~side, as this is the one we will use, and write $\eta=\eta_-$.
\begin{proposition}\cite[Amplification~2.11]{HL}\label{prop_HL_SOD}
	The full subcategories 
	\[\cC_{[0,\eta)}, \cC_{[1,\eta+1)}, \ldots,  \cC_{[l,\eta+l)} \subset \cC_{[0,\eta+l)}\]
	 can be completed to semiorthogonal decompositions
	\begin{align*}
	\cC_{[0,\eta+l)} & = \langle \cC_{[0,\eta)}, \cA_0, \ldots, \cA_{l-1}\rangle \\
	& = \langle \cA_0, \cC_{[1,\eta+1)}, \cA_1, \ldots, \cA_{l-1} \rangle \\
	& = \ldots \\
	& = \langle \cA_0, \cA_1, \ldots, \cA_{l-1}, \cC_{[l,\eta+l)}\rangle,
	\end{align*}
	where
	\[\mathcal{A}_w = \big\{ E \in \dcat(\mathcal{X}/\mathbb{G}_m) \,\big|\, \mathcal{H}^\bullet (i_Z^* E) \text{ have weights in $[w,w+\eta]$, $E$ supported on $S^-$} \big\}.\]
\end{proposition}

Finally, there is an equivalence of $\dcat(Z)$ with $\mathcal{A}_w$. For this, set notation
\[\begin{tikzpicture}[scale=2.5]
\node (X) at (-2.25,0) {$Z$};
\node (0) at (-1.25,0) {$Z/\mathbb{G}_m$};
\node (1) at (0,0)  {$S^-/\mathbb{G}_m$};
\node (2) at (1.25,0) {$\mathcal{X}/\mathbb{G}_m$};
\draw[<-] (X) -- node[below]{$\tau$} (0);
\draw[<-] (0) -- node[below]{$\pi_-$} (1);
\draw[right hook->] (1) -- node[below]{$j$} (2);
\end{tikzpicture}\]
where $\tau$ corresponds to the quotient $\mathbb{G}_m/\mathbb{G}_m\xrightarrow{\sim} 1$.  Then we have the following.
\begin{lemma}\cite[Remark 2.13, Corollary 3.28]{HL}\label{lem_D(Z)_A_omega}
	The functor 
	\begin{equation}
	\label{eqn.embedA}
	\Phi_{\omega}(-) = j_*\pi_-^*(\tau^*(-)\otimes \cO_{Z/\mathbb{G}_m}\langle \omega\rangle)
	\end{equation} 
	is an equivalence of $\dcat(Z)$ with $\cA_{\omega} \subset \dcat(\cX/\mathbb{G}_m)$.	
\end{lemma}

\begin{remark} For comparison with the original paper \cite{HL}, note that there the group acts on a scheme denoted $X$, and $\cX$ is the quotient stack, while for us $\cX$ is the scheme. We denote the fixed locus  by $Z$ as in~\cite{HL}, but we denote the semistable loci by~$\cX^{\pm}$, in~contrast with $X^{ss}$ in~\cite{HL}. Since we consider a torus action, for a one-parameter subgroup~$\lambda$ the `blade' $Y_{\lambda, Z}$ coincides with the stratum  $S_{\lambda, Z}$ by \cite[Remark 2.1]{HL}. We~denote the latter by $S^{\pm}$, depending on whether the one-parameter subgroup $\lambda$ is $\mathbb{G}_m$ itself, or is given by inversion.

The subcategories denoted here by $\mathcal{G}_{\omega}$ and $\cC_{[\omega, \omega + \eta_{\pm})}$ both appear as $\mathcal{G}_{\omega}$ in \cite[Definition~2.8]{HL} and \cite[Lemma~2.9]{HL} respectively. The subcategory which we denote by $\cA_{\omega}$ is $\mathrm{D}^b_{S^-/\mathbb{G}_m}(\cX/\mathbb{G}_m)_{\omega}$ from~\cite[Amplification~3.27]{HL}. \mbox{Indeed,} if condition~(R) is satisfied then \cite[Lemma 2.9]{HL} holds, giving the above definition of $\cA_\omega$. By \cite[Amplifications 3.18 and 3.27]{HL} this subcategory $\mathrm{D}^b_{S^-/\mathbb{G}_m}(\cX/\mathbb{G}_m)_{\omega}$ is equivalent to $\mathrm{D}^b(Z/\mathbb{G}_m)_{\omega}$ which appears in \cite[Amplification~2.11]{HL}.
\end{remark}

\subsection{Derived category of the root stack $\stack[\rootn]$}

We now apply the above theory to the GIT problem~$\cX_\rootn$ from Definition~\ref{def.quotforroot}. We begin by calculating the window widths~$\eta_{\pm}$.

\begin{proposition}\label{prop_window_width}
	 For $\eta_\pm$ the $\mathbb{G}_m$-weight of $\det \mathcal{N}_{S^\pm | \mathcal X}^{\mp 1} |_Z$ we have
\begin{equation*}
\eta_+ = \rootn,\qquad
\eta_- = 1.
\end{equation*}
\end{proposition}
\begin{proof} 
	Over $U_i$, $S^+ = \{y_i =0\}$ and $S^-= \{z_i =0\}$. Hence, locally, $I_{S^+}/I_{S^+}^2$ is spanned by $y_i$, while $I_{S^-}/I_{S^-}^2$ is spanned by $z_i$. The statement follows as the $\mathbb{G}_m$ action is given by $\lambda \cdot y_i = \lambda^{-\rootn}y_i$ and $\lambda \cdot z_i = \lambda z_i$. 
\end{proof}

Base changing the root stack construction to the divisor $D$ itself, we obtain 
\begin{equation}\label{eqn.basechange}
\begin{aligned}
\begin{tikzpicture}[scale=2]
\node (0) at (-1.5,0) {$D$};
\node (3) at (-1.5,1.5) {$\sqrt[\rootn]{\cO_D(D)}$};
\node (1) at (0,0)  {$X$};
\node (2) at (0,1.5) {$\sqrt[\rootn]{X/D}$};
\draw[->] (0) -- node[below]{$i_D$} (1);
\draw[->] (2) -- node[right]{$p$} (1);
\draw[->] (3) -- node[left]{$q$} (0);
\draw[->] (3) -- node[above]{$i$} (2);
\end{tikzpicture}
\end{aligned}
\end{equation}
where $\sqrt[\rootn]{\cO_D(D)}$ is the root stack of the given line bundle on $D$, see Definition \ref{def_root_of_inv_sh}, compare~\cite[Section~5]{IshUed}.
 Indeed, consider $\textrm{Tot}(\cO_D(D) \oplus \cO_D)$ with local fiber coordinates $(z_i,y_i)$. Then the pullback of $\cX_\rootn^+ $  to $D$ is $\cY_\rootn\subset \textrm{Tot}(\cO_D(D) \oplus \cO_D)$ given by $\{y_iz_i^\rootn =0, y_i \neq 0\}$. As $y_i$ is non-zero, $\cY_\rootn$ is isomorphic to the total space of the \mbox{$\mathbb{G}_m$-bundle} associated to $\cO_D(D)$. The $\mathbb{G}_m$ action on $\cY_\rootn$ is fiberwise with weight $-\rootn$. The isomorphism with $\sqrt[\rootn]{\cO_D(D)}$ then follows from Proposition \ref{prop_glob_quot_for_line_bundle}.

\smallskip

We now obtain decompositions for the root stack.

\begin{theorem}\label{thm_SOD_for_r_stack}
	The category $\dcat{\stackbrack(\stack[\rootn]\stackbrack)}$ admits semiorthogonal decompositions
	\begin{equation*}
	\begin{split}
	& \sodbrack\langle p^*\dcat(X)\sodcomma i_*q^*\dcat(D)\sodcomma i_*q^*\dcat(D)\otimes \cO\langle 1\rangle\sodcomma\ldots\sodcomma i_*q^*\dcat(D)\otimes \cO\langle \rootn-2\rangle\sodbrack\rangle \\
	 & = \sodbrack\langle i_*q^*\dcat(D)\sodcomma  p^*\dcat(X)\otimes \cO\langle 1 \rangle\sodcomma i_*q^*\dcat(D)\otimes \cO\langle 1\rangle\sodcomma\ldots\sodcomma i_*q^*\dcat(D)\otimes \cO\langle  \rootn-2\rangle \sodbrack\rangle \\
	& = \ldots \\
	& = \sodbrack\langle i_*q^*\dcat(D)\sodcomma   i_*q^*\dcat(D)\otimes \cO\langle 1\rangle\sodcomma\ldots\sodcomma i_*q^*\dcat(D)\otimes \cO\langle  \rootn-2\rangle\sodcomma p^*\dcat(X)\otimes \cO\langle \rootn-1 \rangle\sodbrack\rangle.
	\end{split}
	\end{equation*}
\end{theorem}
\begin{proof} 
 We check that under the equivalences $\cA_\omega\simeq \dcat(Z) \simeq \dcat(D)$, see Lemma~\ref{lem_D(Z)_A_omega} and Proposition~\ref{prop_quotients}, and $i_+^*\colon \mathcal{C}_{[0,\rootn)} \xrightarrow{\sim}\dcat{\stackbrack(\stack[\rootn]\stackbrack)}$, see  Propositions~\ref{prop_HL_equiv} and~\ref{prop_window_width}, the required semiorthogonal decompositions are the decompositions given by Proposition~\ref{prop_HL_SOD} with $\eta=\eta_-=1$ and $l=\rootn-1$, after noting that $\mathcal{C}_{[0, \rootn)} = \mathcal{C}_{[0, \eta +\rootn-1)}$.
 
 We first check that conditions $(A)$, $(R)$ and $(L)$ of Proposition \ref{prop_HL_equiv} are satisfied. By Proposition \ref{prop_quotients}, over $U_i$ we have $S^+=\{y_i = 0\}$, $S^-=\{z_i = 0\}$ and $Z = \{y_i, z_i =0\}$. Furthermore, $S^+ \simeq \textrm{Tot}(\cO_D)$ and $S^- \simeq \textrm{Tot}(\cO_D(D))$ are clearly locally trivial bundles of affine spaces over $Z$, and the embeddings $S^+ \to \cX$ and $S^- \to \cX$ given by~$y_i$ and~$z_i$ respectively are regular. Finally, by Proposition~\ref{prop_window_width}, $\det \mathcal{N}^{\mp 1}_{S^\pm|\cX}|_Z$ has positive weights~$\rootn$ and~1 respectively.
 
 \smallskip
 \emph{Embeddings of $\dcat(D)$.} Recall from Lemma \ref{lem_D(Z)_A_omega} the embedding $\Phi_{\omega}$ of $\dcat(Z)$ into $\dcat(\cX/\mathbb{G}_m)$ with $\mathcal{A}_\omega$ as its essential image. We show that under the isomorphism $Z\simeq D$, $i_+^*\circ \Phi_{\omega}(-)\simeq i_*q^*(-) \otimes \cO_{\stack[\rootn]}\langle \omega \rangle$. 
 Let $S^\circ$ be the open subscheme $S^- - Z$ of~$S^-$. We have a diagram as follows.
 \[
 \begin{tikzpicture}[scale=2]
 \node (0) at (-1.5,0) {$Z/\mathbb{G}_m$};
 \node (1) at (0,0)  {$S^-/\mathbb{G}_m$};
 \node (2) at (1.5,0) {$\mathcal{X}_\rootn/\mathbb{G}_m$};
 \node (0B) at (-1.5,-1) {$Z/\mathbb{G}_m$};
 \node (1B) at (0,-1)  {$S^\circ/\mathbb{G}_m$};
 \node (2B) at (1.5,-1) {$\mathcal{X}^+_\rootn/\mathbb{G}_m$};
 \draw[<-] (0) -- node[above]{$\pi_-$} (1);
 \draw[right hook->] (1) -- node[above]{$j$} (2);
 \draw[<-] (0B) -- node[below]{$\pi^\circ$} (1B);
 \draw[right hook->] (1B) -- node[below]{$j^\circ$} (2B);
 \draw [double equal sign distance] (0) to (0B);
 \draw[right hook->] (1B) -- node[right]{$k$} (1);
 \draw[right hook->] (2B) -- node[right]{$i_+$} (2);
 \end{tikzpicture}
 \]
 By the description of $\mathcal{X}^+_\rootn$ in Proposition~\ref{prop_quotients}, the right-hand square is Cartesian, so by flat base change we have
  \begin{align*} i_+^*\circ \Phi_{\omega}(-) 
  &\simeq i_+^*j_*\pi_-^*(\tau^*(-)\otimes \cO_{Z/\mathbb{G}_m}\langle \omega\rangle) \\
  &\simeq j^\circ_*k^*\pi_-^*(\tau^*(-)\otimes \cO_{Z/\mathbb{G}_m}\langle \omega\rangle) \\
  &\simeq j^\circ_*{\pi^\circ}^*(\tau^*(-)\otimes \cO_{Z/\mathbb{G}_m}\langle \omega\rangle).  
\end{align*}
Tensor-pullback distributivity and the projection formula then give that
  \begin{align*}
i_+^*\circ \Phi_{\omega}(-)  &\simeq j^\circ_*({\pi^\circ}^*\tau^*(-)\otimes \cO_{S^\circ/\mathbb{G}_m}\langle \omega\rangle) \\
  &\simeq j^\circ_*\big({\pi^\circ}^*\tau^*(-)\otimes j^{\circ*}\cO_{\cX^+_\rootn/\mathbb{G}_m}\langle \omega\rangle\big) \\
  & \simeq j^\circ_*{\pi^\circ}^*\tau^*(-)\otimes \cO_{\cX^+_\rootn/\mathbb{G}_m}\langle \omega\rangle \\
 & \simeq j^\circ_* \sigma^*(-) \otimes \cO_{\cX^+_\rootn/\mathbb{G}_m}\langle \omega \rangle
  \end{align*}
 where $\sigma = \tau \,\pi^\circ$ is shown below. \[
 \begin{tikzpicture}[scale=2]
 \node (XB) at (-2.75,-1) {$Z$};
 \node (0B) at (-1.5,-1) {$Z/\mathbb{G}_m$};
 \node (1B) at (0,-1)  {$S^\circ/\mathbb{G}_m$};
 \draw[<-] (XB) -- node[above]{$\tau$} (0B);
 \draw[<-] (0B) -- node[above]{$\pi^\circ$} (1B);
 \draw[<-] (XB) to[bend right=30] node[below]{$\sigma$} (1B);
 \end{tikzpicture}
 \]
 
 Now for us $S^\circ$ is a $\mathbb{G}_m$-bundle over $Z$ with a $\mathbb{G}_m$-action of  weight~$-\rootn$. By Proposition~\ref{prop_quotients}, this is isomorphic to the bundle $\cO_D(D)$ over $D$ with $\mathbb{G}_m$-action of  weight~$-\rootn$ after removing the zero section. In other words we have an isomorphism $S^\circ/\mathbb{G}_m \simeq \sqrt[\rootn]{\cO_D(D)}$, see Proposition \ref{prop_glob_quot_for_line_bundle}. This fits in a commutative diagram as follows.
 \[
 \begin{tikzpicture}[scale=2]
 \node (0) at (-1.25,0) {$D$};
 \node (1) at (0,0)  {$\sqrt[\rootn]{\cO_D(D)}$};
 \node (2) at (1.5,0) {$\sqrt[\rootn]{X/D}$};
 \node (0B) at (-1.25,-1) {$Z$};
 \node (1B) at (0,-1)  {$S^\circ/\mathbb{G}_m$};
 \node (2B) at (1.5,-1) {$\mathcal{X}^+_\rootn/\mathbb{G}_m$};
 \draw[<-] (0) -- node[above]{$q$} (1);
 \draw[right hook->] (1) -- node[above]{$i$} (2);
 \draw[<-] (0B) -- node[below]{$\sigma$} (1B);
 \draw[right hook->] (1B) -- node[below]{$j^\circ$} (2B);
 \draw[->] (0B) to node[above,sloped]{\small $\sim$} (0);
 \draw[->] (1B) -- node[above,sloped]{\small $\sim$} (1);
 \draw[->] (2B) -- node[above,sloped]{\small $\sim$} (2);
 \end{tikzpicture}
 \]
 We thence  get the required description of the embeddings of $\dcat(D)$, namely
 \[ i_+^*\circ \Phi_\omega(-) \simeq i_*q^*(-)\otimes \cO_{\stack[\rootn]}\langle \omega \rangle. \]
 
  \smallskip
 \emph{Embeddings of $\dcat(X)$.} 
First, we find an inverse of the equivalence $i_-^* \colon \mathcal{C}_{[\omega, \omega +1)} \xrightarrow{\sim} \dcat(X)$ of Proposition \ref{prop_HL_equiv}.
Consider a commutative diagram as below, following the argument of \cite[Lemma~5.2(1)]{CIJS}.
 \begin{equation*}
 \begin{tikzpicture} [xscale=1.25,scale=1.6] 
 \node (A) at (-2.25,0) {$X$};
 \node (Ap) at (-1,0) {$\mathcal{X}^-_\rootn/\mathbb{G}_m$}; 
 \node (C) at (2.25,0) {$\sqrt[\rootn]{X/D}$}; 
 \node (Cp) at (1,0) {$\mathcal{X}^+_\rootn/\mathbb{G}_m$}; 
 \node (B) at (0,-1) {$X$};
 \node (D) at (0,1) {$\mathcal{X}_\rootn/\mathbb{G}_m$}; 
 \draw [double equal sign distance] (A) to (B);
 \draw [->] (D) to node[pos=0.4,right] {$\rho$} (B);
 \draw [->] (C) to node[below right] {$p$} (B);
 \draw [<-] (C) to node[above] {\small $\sim$} (Cp);
 \draw [<-] (A) to node[above] {\small $\sim$} (Ap);
 \draw [->] (Ap) to node[above right] {$\rho_-$} (B);
 \draw [->] (Cp) to node[above left] {$\rho_+$} (B);
 \draw [right hook->] (Ap) to node[above left] {$i_-$} (D);
 \draw [left hook->] (Cp) to node[above right] {$i_+$}  (D);
 \end{tikzpicture}
 \end{equation*}
 Recall that $\cT_\rootn = \textrm{Tot}(\cO_X(D) \oplus \cO_X)$. The morphism $\rho$ decomposes into an inclusion of a Cartier divisor $i_{\cX}\colon \cX_\rootn/\mathbb{G}_m \to \cT_\rootn/\mathbb{G}_m$ followed by a flat projection $\cT_\rootn/\mathbb{G}_m \to X$. It follows that $\rho^*$ is a functor $\dcat(X) \to \dcat(\cX_\rootn/\mathbb{G}_m)$. Indeed, given an object $E$ in $\dcat(\cT_\rootn/\mathbb{G}_m)$ two of the three terms of the functorial exact triangle \eqref{eqtn_cotw_for_divisor} are objects of $\dcat(\cT_\rootn/\mathbb{G}_m)$, hence so is the third one $i_{\cX*} i_{\cX}^* E$. As $i_{\cX*}$ has no kernel, we conclude  that $i_{\cX}^*E \in \dcat(\cX_\rootn/\mathbb{G}_m)$. Further, $\rho_-$ 
 is an isomorphism, hence $i_-^*\rho^* \simeq \rho_-^* \simeq \textrm{Id}_{\dcat(X)}$ by the left-hand side of the diagram above. In other words, $\rho^*$~is an inverse of $i_-^* \colon  \mathcal{C}_{[0,1)} \xrightarrow{\sim} \dcat(X)$. It follows that $\rho^*(-)\otimes \cO_{\cX_\rootn/\mathbb{G}_m}\langle \omega \rangle$ is an inverse of $i_-^* \colon \mathcal{C}_{[\omega, \omega +1)} \xrightarrow{\sim} \dcat(X)$.
 
Using Proposition \ref{prop_HL_equiv} again, we thence have a functor $\dcat(X) \to \dcat{\stackbrack(\stack[\rootn]\stackbrack)}$ with essential image $i_+^* \mathcal{C}_{[\omega, \omega +1)}$ given by
 \[ i_+^*(\rho^*(-) \otimes \cO_{\cX_\rootn/\mathbb{G}_m}\langle \omega \rangle) \simeq i_+^*\rho^*(-) \otimes \cO_{\stack[\rootn]}\langle \omega \rangle \simeq p^*(-) \otimes \cO_{\stack[\rootn]}\langle \omega \rangle. \]
Here the last isomorphism uses the right-hand side of the diagram above, and we find the required description of the embeddings of $\dcat(X)$.
 \end{proof}

\section{Periodic semiorthogonal decompositions}

The following is a preparation for the proof of our main theorem.

\begin{proposition}\label{prop_intertwinement}
The autoequivalence \[- \otimes \cO_{\stack[\rootn]} \langle \rootn \rangle \] of $\dcat{\stackbrack(\sqrt[\rootn]{X/D}\stackbrack)} $ preserves the decompositions of Theorem \ref{thm_SOD_for_r_stack}. Furthermore, it intertwines  with autoequivalences
\[- \otimes \cO_X(-D) \qquad \text{and} \qquad - \otimes\, \cO_D(-D)\] via the embeddings $p^*$ from $\dcat{(X)}$, and $i_* q^*$ from $\dcat{(D)}$, respectively.
\end{proposition}

\begin{proof}
	 Note that $\cO_{\stack[\rootn]} \langle \rootn \rangle \simeq p^* \cO_X(-D)$ by Remark \ref{rem_O(-r)}. Using tensor-pullback distributivity we have
\begin{align*}
\big(- \otimes \cO_{\stack[\rootn]} \langle \rootn \rangle\big) \,p^* & \simeq (- \otimes p^* \cO_X(-D)) \,p^* \\
& \simeq p^* (- \otimes \cO_X(-D)). 
\end{align*}
Furthermore using the projection formula and commutativity of square~\eqref{eqn.basechange} gives
\begin{align*}
\big(- \otimes \cO_{\stack[\rootn]} \langle \rootn \rangle\big) \,i_* q^* & \simeq i_* \big(- \otimes \,i^* \cO_{\stack[\rootn]} \langle \rootn \rangle\big) \,q^* \\
 & \simeq i_* (- \otimes i^* p^* \cO_X(-D)) \,q^* \\
 & \simeq i_* (- \otimes q^* i_D^*  \cO_X(-D)) \,q^* \\
& \simeq i_* q^*(- \otimes i_D^*  \cO_X(-D)) \\
& \simeq i_* q^*(- \otimes \cO_D(-D)). 
\qedhere \end{align*}
\end{proof}

Recall \cite{Bon} that given a semiorthogonal decomposition 
\begin{equation}\label{eqtn_SOD}
\cC = \sodbrack \langle \cA, \cB \sodbrack \rangle
\end{equation} 
of a triangulated category $\cC$ with $\cA$ and $\cB$ admissible the left and right dual semiorthogonal decompositions, $\sodbrack \langle L_{\cA}\cB, \cA \sodbrack \rangle$ and $\sodbrack \langle \cB, R_{\cB}\cA \sodbrack \rangle$ respectively, exist  for the left mutation $L_{\cA}\cB$ of $\cB$ over $\cA$ and the right mutation $R_{\cB}\cA$ of $\cA$ over $\cB$. We say that the semiorthogonal decomposition \eqref{eqtn_SOD} is \emph{strongly admissible} if the $\nth[N]$ left and right dual decompositions exist for any $N$. Given such a decomposition the Artin braid group on two strands $B_2 \cong \mathbb{Z}$ acts on the set of the decomposition \eqref{eqtn_SOD} and all its left and right duals.
\begin{definition}\label{def_periodic_SOD}\cite[Section 4.2]{DKS}
	A semiorthogonal decomposition \eqref{eqtn_SOD} of a triangulated category $\cC$ is \emph{$N$-periodic} if the $\nth[N]$ right dual decomposition is again \eqref{eqtn_SOD}, i.e.\ if the decomposition is strongly admissible and the action of $\mathbb{Z}$  factors through $\mathbb{Z}_N$.
\end{definition}

Our main theorem now follows rapidly.

\begin{theorem}\label{thm_2r_period}
	Take a full subcategory
	\[\mathcal{D} = \sodbrack\langle i_*q^*\dcat(D)\sodcomma i_*q^*\dcat(D) \otimes \cO\langle 1\rangle\sodcomma \ldots\sodcomma i_*q^*\dcat(D) \otimes \cO\langle \rootn-2 \rangle \sodbrack\rangle \subset \dcat{\stackbrack(\stack[\rootn]\stackbrack)}.\] 
	Then the following semiorthogonal decomposition is $2\rootn$-periodic.
	\begin{equation}\label{eqtn_2r_per_SOD}
	\dcat{\stackbrack(\stack[\rootn]\stackbrack)} = \langle p^*\dcat(X), \mathcal{D} \rangle
	\end{equation}
\end{theorem}
\begin{proof}
	By Theorem \ref{thm_SOD_for_r_stack}, the decomposition right dual to $\langle p^*\dcat(X), \mathcal{D} \rangle$ is 
	\[\sodbrack\langle \mathcal{D}\sodcomma p^*\dcat(X)\otimes \cO_{\stack[\rootn]}\langle \rootn-1 \rangle  \sodbrack\rangle\]
	and so the second right  dual is  \[\sodbrack\langle p^*\dcat(X)\otimes \cO_{\stack[\rootn]}\langle \rootn-1 \rangle\sodcomma \mathcal{D} \otimes \cO_{\stack[\rootn]} \langle \rootn-1 \rangle\sodbrack\rangle,\]
	i.e.\ the twist of the original decomposition by $\cO_{\stack[\rootn]} \langle \rootn-1 \rangle$. It follows that the $\nth[2k]$~right~dual is the twist of the original decomposition by~$\cO_{\stack[\rootn]}\langle k(\rootn-1) \rangle$. But then, using Proposition~\ref{prop_intertwinement}, the $\nth[2\rootn]$ right dual decomposition is the original decomposition, namely $\langle p^*\dcat(X), \mathcal{D} \rangle$.
\end{proof}

Finally we describe the gluing functors for the periodic decomposition above.

\begin{proposition}\label{prop_gluing_functor}
	The gluing functor $\mathcal{D} \to \dcat(X)$ for the decomposition \eqref{eqtn_2r_per_SOD} after restriction to $ i_*q^*\dcat(D)\otimes \cO\langle k \rangle \subset \mathcal{D}$ is $i_{D*}[1]$ for $k=0$, and zero otherwise.
\end{proposition}
\begin{proof}
	Write $\sodembed_{\mathcal{D}}$ for the embedding $\mathcal{D} \to \dcat{\stackbrack(\stack[\rootn]\stackbrack)}$. Then the gluing functor for~\eqref{eqtn_2r_per_SOD} is $p_* \sodembed_{\mathcal{D}}[1]$ by \cite[Section 2.2]{KuzLun}. As
	\[ p_*\big(-\otimes \cO_{\stack[\rootn]}\langle k \rangle\big)i_*q^*\simeq p_*i_*\big(q^*(-) \otimes \cO_{\sqrt[\rootn]{\cO_D(D)}}\langle k \rangle\big) \simeq i_{D*}q_*\big(q^*(-) \otimes \cO_{\sqrt[\rootn]{\cO_D(D)}}\langle k \rangle\big),\]
	the statement follows from 
	\[
	q_*\big(q^*(-)\otimes \cO_{\sqrt[\rootn]{\cO_D(D)}}\langle k \rangle\big) = \left\{\begin{array}{ll}\textrm{Id} & \textrm{for } k=0\\
	0 & \textrm{for } k=1, \ldots, \rootn-2. \end{array} \right.
	\]
	Local calculations of $\mu_\rootn$-invariants show that $q_*q^* \simeq  \textrm{Id}_{\dcat(D)}$. The claimed vanishing for $k =1, \ldots, \rootn-2$ follows from the mutual orthogonality  of the essential images of $q^*(-)\otimes \cO_{\sqrt[\rootn]{\cO_D(D)}}\langle k \rangle$  for $k = 0, \ldots, \rootn-2$ in $\dcat{\stackbrack(\sqrt[\rootn]{\cO_D(D)}\stackbrack)}$,
	see Remark \ref{rem_O_r_for_inv_sh} and the proof  of \cite[Theorem 1.5]{IshUed}.
\end{proof}

As already discussed in Section~\ref{sec.vgit}, T.~Dyckerhoff, M.~Kapranov and V.~Schechtman consider $N$-periodic semiorthogonal decompositions of stable infinity categories as well as $N$-spherical functors of such categories~\cite{DKS}. In particular, a 4-spherical functor is an analogue of a spherical functor of DG categories. In \cite[Theorem~4.2.1]{DKS} the authors prove that a functor is $N$-spherical if and only if it is the gluing functor for an $N$-periodic semiorthogonal decomposition. Motivated by their work we give the following.
\begin{definition}
	The gluing functor  $\sodembed_{\mathrm{A}}^R\sodembed_{\mathrm{B}}$ of an $N$-periodic semiorthogonal decomposition \eqref{eqtn_SOD} of a triangulated category $\cC$ is \emph{$N$-triangle-spherical}.
\end{definition}
When $N$ is equal to 4, we simply say that the functor is \emph{triangle-spherical}.

\begin{corollary}\label{cor_div_is_sph} 
	The functor $i_{D*}\colon \dcat{(D)}\to\dcat{(X)}$ is triangle-spherical. The unit and the counit for the $i_D^*\dashv i_{D*}$ adjunction fit into functorial exact triangles \eqref{eqtn_cotw_for_divisor} and \eqref{eqtn_tr_for_divisor}.
\end{corollary}

\begin{proof} 
	Consider the 4-periodic semiorthogonal decomposition \eqref{eqtn_main_SOD}. By definition, the gluing functor $\sodembed_{\mathrm{A}}^R\sodembed_{\mathrm{B}}$ is~triangle-spherical, and $\sodembed_{\mathrm{A}}^R\sodembed_{\mathrm{B}} = p_*i_*q^* \simeq i_{D*}$ (see the proof of Proposition~\ref{prop_gluing_functor}). This gluing functor has left adjoint $\sodembed_{\mathrm{B}}^L\sodembed_{\mathrm{A}} \simeq i_D^*$. By \eqref{eqtn_cotwist_for_4-per} and \eqref{eqtn_twist_for_4-per}, the cones of the $i_D^*\dashv i_{D*}$ adjunction unit and counit are compositions of right mutations.
	
	For the pair of semiorthogonal decompositions 
	\begin{equation}\label{eqtn_1_2_main_SOD}
	\dcat{\stackbrack(\stack \stackbrack)}  = \sodbrack\langle p^*\dcat(X) \sodcomma  i_*q^*\dcat(D) \sodbrack\rangle = \sodbrack \langle i_*q^*\dcat(D) \sodcomma p^*\dcat(X)\otimes \cO_{\stack}\langle 1 \rangle \sodbrack \rangle
	\end{equation}
   the right mutation of $p^*\dcat{(X)}$ is, up to shift, given  by $-\otimes \cO_{\stack}\langle 1 \rangle \colon p^*\dcat{(X)} \to p^*\dcat{(X)} \otimes \cO_{\stack}\langle 1 \rangle$. Indeed, the functor is an equivalence of the right and the left 
	orthogonal complements to $i_*q^*\dcat{(D)}$. In order to determine the shift, we evaluate $\sodembed_{\mathrm{C}}^R\sodembed_{\mathrm{A}} \simeq p_*(p^*(-)\otimes \cO_{\stack}\langle -1 \rangle)$ at $\cO_X$. As $\cO_X$ is locally free and the morphism $\cV\to X$ in Proposition \ref{prop_global_dscrp_of_root} is affine, $p_*(p^*\cO_X \otimes \cO_{\stack}\langle -1 \rangle)\in \textrm{Coh}(X)$. It follows that $-\otimes \cO_{\stack}\langle 1 \rangle$ is the mutation functor for \eqref{eqtn_1_2_main_SOD} . It is also the mutation for \eqref{eqtn_1_2_main_SOD} twisted by $\cO_{\stack}\langle 1 \rangle$. Hence, $-\otimes \cO_{\stack}\langle 2 \rangle$ is the composition of mutations between $p^*\dcat{(X)}$ and its fourth left orthogonal complement in $\dcat{\stackbrack(\stack\stackbrack)}$. By Proposition \ref{prop_intertwinement}, it is the image under $p^*$ of the equivalence $-\otimes \cO_X(-D)\colon \dcat{(X)} \to \dcat{(X)}$. We conclude that $\sodembed_{\mathrm{A}}^R{\sodembed_{\mathrm{C}}}\sodembed_{\mathrm{C}}^R{\sodembed_{\mathrm{A}}} \simeq - \otimes \cO_X(-D)$ and \eqref{eqtn_cotw_for_divisor} follows.
	
	Analogously, for the pair of semiorthogonal decompositions
	\begin{equation}\label{eqtn_2_3_main_SOD}
	\begin{aligned}
	\dcat{\stackbrack(\stack \stackbrack)}   &= \sodbrack \langle i_*q^*\dcat(D) \sodcomma p^*\dcat(X)\otimes \cO_{\stack}\langle 1 \rangle \sodbrack \rangle \\
	&= \sodbrack\langle p^*\dcat(X) \otimes \cO_{\stack} \langle 1 \rangle \sodcomma  i_*q^*\dcat(D) \otimes \cO_{\stack} \langle 1 \rangle \sodbrack\rangle
	\end{aligned}
	\end{equation}
	 the right mutation of $i_*q^*\dcat{(D)}$ is, up to shift, given by $-\otimes \cO_{\stack} \langle 1 \rangle$. To determine the shift, we check that $\sodembed_{\mathrm{B}}^R\sodembed_{\mathrm{D}} \simeq q_*i^!(i_*q^*(-)\otimes \cO_{\stack}\langle 1 \rangle)$ maps $\cO_D$ to an object of $\textrm{Coh}(D)[-1]$. We conclude that $-\otimes \cO_{\stack}\langle 1 \rangle[1]$ is the mutation functor for \eqref{eqtn_2_3_main_SOD}. It~is also the mutation for \eqref{eqtn_2_3_main_SOD} twisted by $\cO_{\stack}\langle 1 \rangle$. Hence, $-\otimes \cO_{\stack}\langle 2 \rangle[2]$ is the composition of mutations between $i_*q^*\dcat{(D)}$ and its fourth left orthogonal complement in $\dcat{\stackbrack(\stack\stackbrack)}$. By Proposition \ref{prop_intertwinement}, it is the image under $i_*q^*$ of the equivalence \mbox{$-\otimes \cO_D(-D)[2]\colon \dcat{(D)} \to \dcat{(D)}$}. We conclude that $\sodembed_{\mathrm{B}}^L{\sodembed_{\mathrm{D}}} \sodembed_{\mathrm{D}}^L{\sodembed_{\mathrm{B}}} \simeq - \otimes \cO_D(-D)[2]$ and~\eqref{eqtn_tr_for_divisor} follows.
\end{proof}

\opt{no-manuscripta}{\newpage}
\opt{manuscripta}{\noindent \emph{No datasets were generated or analysed during the current work. On behalf of all authors, the corresponding author states that there is no conflict of interest.}}

\bibliographystyle{alpha}
\bibliography{ref}

\end{document}